\documentclass[12pt]{article}
{\begin{Sbox}\begin{minipage}}%
{\end{minipage}\end{Sbox}\fbox{\TheSbox}}%

\usepackage[psamsfonts]{amssymb}
\usepackage{amsmath, amsthm, anysize, enumerate, color, url}
\usepackage{graphicx}
\usepackage{subfigure}

\newtheorem{theorem}{Theorem} %%%[section]

\newtheorem{lemma}{Lemma} %%%[theorem]
\newtheorem{proposition}{Proposition} %%%[theorem]

\usepackage[colorlinks,
                   linkcolor=red,
                  anchorcolor=blue,
                  citecolor=blue
                  ]{hyperref}

\theoremstyle{definition}
\newtheorem{remark}{Remark}  %%%[theorem]

\newcommand{\E}{\mathbb{E}}
\newcommand{\R}{\mathbb{R}}

\renewcommand{\P}{\mathbb{P}}

\newcommand{\bs}{\boldsymbol}

\usepackage{setspace}

\begin{document}

\title{Asymptotic normality in the maximum entropy models on graphs with an increasing number of parameters }

\author{Ting Yan \thanks{Department of Statistics, Central China Normal University, Wuhan, 430079, China.
\texttt{Email:} tingyanty@gmail.com.} \hspace{20mm}
Yunpeng Zhao \thanks{Department of statistics,
George Mason University, Virginia 22030-4444,
USA. \texttt{Email:} yzhao15@gmu.edu.} %%%Acknowlegements
\hspace{20mm}
Hong Qin \thanks{Department of Statistics, Central China Normal University, Wuhan, 430079, China.
\texttt{Email:} qinhong@mail.ccnu.edu.cn. } %%%Acknowledgements
\\
$^{*,\ddag}$Central China Normal University and $^\dag$George Mason University
}
\date{}
% \date \today

\maketitle

\begin{abstract}
Maximum entropy models, motivated by applications in neuron science, are natural generalizations of the $\beta$-model to weighted graphs. Similar to the $\beta$-model, each vertex in maximum entropy models is assigned a potential parameter, and the degree sequence is the natural sufficient statistic. Hillar and Wibisono (2013) has proved the consistency of the maximum likelihood estimators. In this paper, we further establish the asymptotic normality for any finite number of the maximum likelihood estimators in the maximum entropy models with three types of edge weights, when the total number of parameters goes to infinity. Simulation studies are provided to illustrate the asymptotic results.

%Maximum entropy distributions on graphs, in which the degree sequence is the natural sufficient statistic
%and each vertex is assigned a potential parameter, are widely used in neuron networks.
%Hillar and Wibisono (2013) have proved that the maximum likelihood estimators
%are uniformly consistent when the number of parameters goes to infinity.
%In this paper, we further establish its asymptotic normality.
%Simulations are provided to illustrate the asymptotic results.

\end{abstract}

\vskip 5 pt

Running title: Asymptotic normality in the maximum entropy models

\section{Introduction}

In neuron networks, neurons in one region of the brain may transmit a continuous signal using
sequences of spikes to a second receiver region. The coincidence detectors in the second region
capture the absolute difference in spike times between pairs of neurons projecting
from the first region. There may be three possible types of timing differences: zero or nonzero indicator;
countable number of possible values; any nonnegative real value. Exploring how the transmitted
signal in the first region can be recovered by the second is a basic question in the analysis of neuron networks.
Maximum entropy models provide a possible solution to this question for the above three possible weighted edges.
For detailed explanations, see \cite{Hillar:Wibisono:2013}; for their wide applications in the biological studies as well as other
disciplines such as
economics and physics, see \cite{Hillar:Wibisono:2013, Duik:Phillips:Schapire:2007, Banavar:Maritan:Volkov:2010, Wu:2003, Yeo:Bergue:2004}
and references therein.
Maximum entropy models (sometimes with different names) also appear in other fields of network analysis, e.g., community detection and social network analysis. For example, see \cite{Chatterjee, Bickel:Chen:2009, Bickel:Chen:Levina:2012, Olhede:Wolfe:2012, Zhao:Levina:Zhu:2012, Rinaldo2013}.
%% Chatterjee, Diaconis and Sly (2011); Bickel and Chen (2009); Bickel, Chen and Levina (2012);
%%Olhede and Wolfe (2012); Zhao, Levina and Zhu (2012); Rinaldo, Petrovi\'{c} and Fienberg (2013).}

In the maximum entropy models, the degree sequence is the exclusively natural sufficient
statistics on the exponential family distributions and fully captures the information of an undirected graph. Its study primarily focuses on understanding the generating mechanisms of networks.
%and has appeared in many literatures (Chatterjee, Diaconis and Sly, 2011; Bickel and Chen, 2009; Bickel, Chen and Levina, 2012;
%Olhede and Wolfe, 2012; Zhao, Levina and Zhu, 2012; Rinaldo, Petrovi\'{c} and Fienberg, 2013).
When network edge takes dichotomous values (``0" or ``1"), the maximum entropy model becomes the $\beta$-model (a name given by
Chatterjee, Diaconis and Sly (2011)), an undirected version of the $p_1$ model for directed graphs by Holland and Leinhardt (1981).
Rinaldo, Petrovi\'{c} and Fienberg (2013)
derived necessary and sufficient conditions for the existence and uniqueness of the maximum likelihood estimate (MLE).
As the number of parameters goes to infinity,
Chatterjee, Diaconis and Sly (2011) proved that the MLE
is uniformly consistent; Yan and Xu (2013) further derived its asymptotical normality.
When the maximum entropy models involve the finite discrete, infinite discrete or continuous weighted edges,
Hillar and Wibisono (2013) have obtained the explicit conditions for the existence and uniqueness of the MLE
and proved that the MLE is uniformly consistent as the number of parameters goes to infinity.

Statistical interests are involved with not only the consistency of estimators but also its asymptotic distributions.
The latter can be used to construct the confidence interval on parameters and performed the hypothesis testing.
In the asymptotic framework considered in this paper, the number of network vertices goes to infinity and
the number of parameter is identical to the dimension of
networks (i.e., the number of vertices). Instead of studying a more complicated situation on linear combinations of all
MLEs, we describe the central limit theorems for the MLEs through the asymptotic behavior of a
finite number of the MLEs, although the total number of parameters goes to infinity.
With this point, we aim to establish the asymptotic normality of the MLEs
when edges take three types of weights as in Hillar and Wibisono (2013).
A key step in our proofs applies a highly accurate approximate inverse of the Fisher information matrix
by Yan and Xu (2013).

The remainder of this article is organized as follows. In Section 2, we lay out the
asymptotic distributions of the MLEs in the maximum entropy models with
the finite discrete, infinite discrete and continuous weighted edges in subsection 2.1, 2.2, 2.3,
respectively. Simulation studies are given in Section 3.
Section 4 concludes with summary and discussion.
All proofs are relegated to Appendix.

\section{Asymptotic normalities}
We first give a brief description on the maximum entropy models.
Consider an undirected graph $\mathcal{G}$ with no self-loops on $n$ vertices labeled by ``1, \ldots, n".
Let $a_{ij}$ be the weight of edge $(i,j)$ taking values from the set $\Omega$,
where $\Omega$ could be a finite discrete, infinite discrete or continuous set.
Define $d_i = \sum_{j \neq i} a_{ij}$ as the degree of vertex $i$, and $\mathbf{d} = (d_1, \dots, d_n)^T$ is the degree sequence of $\mathcal{G}$.
Let $\mathcal{S}$ be a $\sigma$-algebra over the set $\Omega$ of all possible values of $a_{ij}$, $1\le i<j\le n$. Assume there is a canonical $\sigma$-finite probability measure $\nu$ on $(\Omega,\mathcal{S})$. Let $\nu^{\binom{n}{2}}$ be the product measure on $\Omega^{\binom{n}{2}}$.
The maximum entropy models assume that the density function of the symmetric adjacent matrix $\mathbf{a} = (a_{ij})_{i,j=1}^n$
with respective to $\nu^{\binom{n}{2}}$ has the exponential form
with the degree sequence as natural sufficient statistics\footnote{Following
Hillar and Wibisono (2013), we use $-\theta$ in the parameterization~\eqref{Eq:MaxEntDist} instead of the classical $\theta$ since it will simplify the notations in the later presentation.}, i.e.,
\begin{equation}\label{Eq:MaxEntDist}
p_\theta(\mathbf{a}) = \exp \big( -\boldsymbol{\theta}^T \mathbf{d} - z(\boldsymbol{\theta}) \big),
\end{equation}
where $z(\boldsymbol{\theta})$ is the normalizing constant,
\begin{eqnarray*}
z(\boldsymbol{\theta}) = \log \int_{S^{\binom{n}{2}}} \exp\big( -\boldsymbol{\theta}^T \mathbf{d} \big) \; \nu^{\binom{n}{2}}(d \mathbf{a})
& = &
\log \prod_{1\le i< j\le n} \int_S \exp \big(-(\theta_i+\theta_j) a_{ij} \big) \; \nu(da_{ij}),
\end{eqnarray*}
and for fixed $n$, the parameter vector $\boldsymbol{\theta} = (\theta_1, \dots, \theta_n)^T$ belongs to the natural parameter space (Page 1, Brown, 1984)
\begin{equation*}
\Theta = \{\boldsymbol{\theta} \in \R^n \colon z(\boldsymbol{\theta}) < \infty\}.
\end{equation*}
The parameters $\theta_1, \dots, \theta_n$ can be interpreted as the strength of each vertex
that determines how strongly the vertices are connected to each other. The probability distribution \eqref{Eq:MaxEntDist} implies
that the edges $(i,j)$ for all $1\le i<j \le n$ are mutually independent.
Since the sample is one realization of a graph, the density function in \eqref{Eq:MaxEntDist} is also the likelihood function.
We can see that
the solution to $-\nabla z(\boldsymbol{\theta}) = \mathbf{d}$ is the maximum likelihood estimator (MLE) of $\theta$.

We now consider the asymptotic distributions of the MLEs as the number of parameters goes to infinity.
Let $V_n=(v_{ij})_{i,j=1,\ldots, n}$
be the Fisher information matrix of the parameters $\theta_1, \ldots, \theta_n$. It can be written as
\[
V_n= \frac{\partial^2 z(\boldsymbol{\theta})}{\partial \boldsymbol{\theta} \partial \boldsymbol{\theta}^T}.
\]
For three common types of weights as introduced in Section 1, $V_n$ is the diagonal
dominant matrix with nonnegative entries. This property is crucially used in the proof of the central limit theorem
on the MLE.

\subsection{Finite discrete weights}
When network edges take finite discrete weights, we assume $\Omega = \{0,1, \ldots, q-1\}$ with $q$ a fixed integer.
In this case, $\nu$ is the counting measure
and the edge weights $a_{ij}$ are independent multinomial random variables with
the probability:
\[
P(a_{ij}=a) = \frac{  e^{a(\theta_i + \theta_j)}  }{ \sum_{k=0}^{q-1} e^{k(\theta_i+\theta_j)} }, ~~a=0, 1, \ldots, q-1.
\]
This model is a direct generalization of the $\beta$-model that only consider the dichotomous edges.
The normalizing constant is
\begin{equation*}
z(\boldsymbol{\theta})=\sum_{1\le i < j \le n} \log \sum_{a=0}^{q-1} e^{-(\theta_i+\theta_j)a},
\end{equation*}
and the parameter space is $\Theta=R^n$.
Let $\bs{\widehat{\theta}}=(\theta_1, \ldots, \theta_n)^T$ be the MLE of $\bs{\theta}=(\theta_1, \ldots, \theta_n)^T$.
The likelihood equations are
\begin{equation}\label{eq:likelihood-finite}
d_i  =  \sum_{j=1; j\neq i}^n \sum_{a=0}^{q-1}
\frac{  ae^{a(\widehat{\theta}_i + \widehat{\theta}_j)}  }{ \sum_{k=0}^{q-1} e^{k(\widehat{\theta}_i+\widehat{\theta}_j)} } ,~~~i=1,\ldots, n, \\
\end{equation}
which is identical to the moment estimating equations.
The fixed point iteration algorithm by Chatterjee, Diaconis and Sly (2011) or the minorization-maximization algorithm by Hunter (2004)
can be used to solve the above system of equations or analogous problems in the next two subsections.
Following Hillar and Wibisono (2013), we assume that $\max_{1\le i\le n}|\theta_i|$
is bounded by a constant when considering the asymptotic distribution of the MLE.
The central limit theorem for the MLE is stated as follows, whose proof is given in Appendix A.

\begin{theorem}\label{Thm:Asym:Binary}
In the case of finite discrete weights,
the diagonal entries of $V_n$ has the following representation:
\[
v_{ii}=\sum_{j=1; j\neq i}^n
\frac{\sum_{0\le k< l\le q-1}(k-l)^2e^{(k+l)(\theta_i+\theta_j)} }{ (\sum_{a=0}^{q-1} e^{a(\theta_i+\theta_j)})^2 },~~i=1, \ldots, n.
\]
Assume that $\max_{1\le i\le n}|\theta_i|$ is bounded by a fixed constant.
Then for any fixed $r\ge 1$, the vector $(v_{11}^{1/2}(\hat{\theta}_1-\theta_1), \ldots,
v_{rr}^{1/2}(\hat{\theta}_r-\theta_r))$ is asymptotically standard
multivariate normal as $n\to\infty$.
\end{theorem}

Notice that $\hat{\theta}_1, \ldots, \hat{\theta}_r$ are asymptotically mutually independent by the above theorem.
This is due to that the maximum entropy models imply the mutually independent edges
in a graph. It further implies that $V_n^{-1}$ is an approximate diagonal matrix shown in Proposition \ref{Prop:Matrix:Inverse} such that
$\hat{\theta}_1, \ldots, \hat{\theta}_r$ are asymptotically mutually independent.

\subsection{Continuous weights}\label{Sec:Weight-R}
When network edges take continuous weights,
$\Omega = [0, \infty)$, $\nu$ is the Lebesgue measure on $[0, \infty)$ and the normalizing constant is
\begin{equation*}
z(\boldsymbol{\theta})=-\sum_{1\le i< j\le n} \log (\theta_i+\theta_j).
\iffalse
z_1(t) = \log \int_{\R_0} \exp(-ta) \: da =
\begin{cases}
\log(1/t) & \text{ if } t > 0 \\
\infty \quad & \text{ if } t \le q 0.
\end{cases}
\fi
\end{equation*}
Therefore, the corresponding natural parameter space is
\begin{equation*}
\Theta = \{(\theta_1, \dots, \theta_n) \in \R^n \colon \theta_i+\theta_j > 0 \text{ for } i \neq j\}.
\end{equation*}
The edge weights $a_{ij}$ ($1\le i< j\le n$) are independently distributed by exponential distributions with density
\begin{equation*}
p(a ) = (\theta_i+\theta_j) \: \exp\big(-(\theta_i+\theta_j) a \big),~~a>0,~~1\le i< j\le n,
\end{equation*}
whose expectation is $(\theta_i+\theta_j)^{-1}$.
The likelihood equations are
\begin{equation*}
d_i= \sum_{j\neq i} \frac{1}{ \hat{\theta}_i + \hat{\theta}_j },~~~~i=1,\ldots, n.
\end{equation*}
The asymptotic distribution of the MLE is stated as follows, whose proof is given in Appendix B.

\begin{theorem}\label{Thm:Asym:Cont}
Let $L_n = \min_{i \neq j} (\theta_i + \theta_j)>0$ and $M_n = \max_{i \neq j} (\theta_i + \theta_j)$.
In the case of continuous weights, the diagonal entries of $V_n$ are:
\[
v_{ii}=\sum_{j=1,j\neq i}^n \frac{ 1 }{(\theta_i + \theta_j)^2},~~i=1,\ldots,n.
\]
If $M_n/L_n=o\{ n^{1/16}/(\log n)^{1/8} \}$, then for any fixed $r\ge 1$, the vector $(v_{11}^{1/2}(\hat{\theta}_1-\theta_1), \ldots,
v_{rr}^{1/2}(\hat{\theta}_r-\theta_r))$ is asymptotically standard multivariate normal as
$n\to\infty$.
\end{theorem}

\subsection{Infinite discrete weights}
When edges take infinite discrete weights, we assume that $\Omega = \{0, 1, \ldots \}$.
In this case, $\nu$ is the counting measure, the normalizing constant is
\begin{equation*}
z(\boldsymbol{\theta})= \sum_{1\le i < j \le n} \log \sum_{a = 0}^\infty \exp(-(\theta_i+\theta_j)a)
=\sum_{1\le i < j \le n}-\log\big( 1-\exp(-(\theta_i+\theta_j)),
\iffalse
z_1(t) = \log \sum_{a = 0}^\infty \exp(-ta) =
\begin{cases}
-\log\big( 1-\exp(-t) \big) & \text{ if } t > 0 \\
\infty \quad & \text{ if } t \leq 0.
\end{cases}
\fi
\end{equation*}
and the natural parameter space is
\begin{equation*}
\Theta = \{(\theta_1,\dots,\theta_n) \in \R^n \colon \theta_i+\theta_j > 0 \text{ for } i \neq j\}.
\end{equation*}
The edge weights $a_{ij}$ are independent geometric random variables with probability mass function:
\begin{equation*}
\P^\ast(a_{ij}=a) = \big(1-\exp(-\theta_i-\theta_j)\big) \: \exp\big( -(\theta_i+\theta_j) a \big), ~~1\le i< j\le n.
\end{equation*}
\iffalse
whose expectations are
\begin{equation*}
\E a_{ij} = \frac{\exp(-\theta_i-\theta_j)}{1-\exp(-\theta_i-\theta_j)} = \frac{1}{\exp(\theta_i+\theta_j)-1}.
\end{equation*}
\fi
The likelihood equations are
\begin{equation*}
d_i = \sum_{j \neq i} \frac{1}{\exp(\hat{\theta}_i+\hat{\theta}_j)-1}, ~~ i = 1,\dots,n,
\end{equation*}
which are identical to the moment estimating equations.

\begin{theorem}\label{Thm:ConsistencyDisc}
Let $L_n = \min_{i \neq j} (\theta_i + \theta_j)>0$ and $M_n = \max_{i \neq j} (\theta_i + \theta_j)$.
In the case of infinite discrete weights, the diagonal entries of $V_n$ are
\[
v_{ii} =\sum_{j=1;j\neq i}^n \frac{e^{\theta_i+\theta_j}}{(e^{\theta_i+\theta_j}-1)^2 }, ~~i=1,\ldots, n.
\]
If $e^{17M_n}/ L_n^3 =o\{n^{1/2}/\log n \}$, then for any fixed $r\ge 1$, the vector $(v_{11}^{1/2}(\hat{\theta}_1-\theta_1), \ldots,
v_{rr}^{1/2}(\hat{\theta}_r-\theta_r))$ is asymptotically standard
multivariate normal as $n\to\infty$.
\end{theorem}

\begin{remark}
By Theorems \ref{Thm:Asym:Binary}--\ref{Thm:ConsistencyDisc}, we have: (1) for any fixed $r$,
$\hat{\theta}_1$, $\ldots$, $\hat{\theta}_r$ are asymptotically independent; (2) as $n\rightarrow \infty$, the
convergence rate of $\hat{\theta}_i$ is $1/v_{ii}^{1/2}$. If $M_n$ and $L_n$ are constants, then this convergence rate is
in the magnitude of $n^{-1/2}$; otherwise it is
between $O\{1/(n^{1/2}L_n)\}$ and  $O\{1/(n^{1/2}M_n)\}$ when edges take continuous weights,
between $O\{(e^{M_n}-1)/[n^{1/2}e^{-M_n/2}]\}$ and $O\{(e^{L_n}-1)/[n^{1/2}e^{-L_n/2}]\}$ when edges
take infinite discrete weights. To compare with the convergence rate in the continuous and infinite discrete cases,
we consider a special case $M_nL_n=1$.
Since $e^{M_n/2}(e^{M_n}-1)\gg M_n=L_n^{-1}$  and $(e^{L_n}-1)e^{L_n/2}\sim L_n =M_n^{-1}$
when $M_n$ is large enough, the former is faster than the latter.
This can be understood that a lower convergence rate can be incurred if the parameter vector is
more quickly close to the boundary of the mean parameter space by noting that $\E(d_i)=\sum_{j\neq i} (\theta_i+\theta_j)^{-1}$ in the continuous case and $\E(d_i)=\sum_{j\neq i}(e^{\theta_i+\theta_j}-1)^{-1}$ in the infinite discrete case.
\end{remark}

\begin{remark}
In contrast with the conditions (i.e., $M_n^2/L_n=o(n^{1/2}/(\log n)^{1/2})$ in the continuous case and
$e^{5M_n}(e^{L_n/2}-1)^{-1/2}=o(n^{1/2}/(\log n)^{1/2})$ in the infinite discrete case)
guaranteeing the consistency of the MLE by Hillar and Wibisono (2013),
the ones for asymptotic normality seems much more strict.
The simulations in the next section suggest there may be space for improvement.
On the other hand, the consistency and asymptotic normality
for the MLE in the finite discrete case requires the assumption that all parameters are bounded by a constant.
This assumption may not be best possible. We will investigate these problems in the future.
\end{remark}

\begin{remark}
The three theorems in this section only describe the joint asymptotic distribution of
the first $r$ estimators $\hat{\theta}_1, \ldots, \hat{\theta}_r$ with
a fixed constant $r$. Actually, the starting point of subscripts is not essential. These three theorems hold for any
fixed $r$ MLEs. Since the usual counting subscript starts from $1$, we only show the case presented in the theorems.
Our proofs can be directly extended to the case of any $r$ fixed MLEs $\hat{\theta}_{i_1}, \ldots, \hat{\theta}_{i_r}$
without any difficulty.
Another interesting problem is investigating the asymptotic distribution of the linear combination $\sum_{i=1}^n c_i\hat{\theta}_i$
on all the MLEs or a linear combination with a growing number of the MLEs as pointed by one referee.
Are there results similar to Propositions \ref{pro:central:binary}--\ref{pro:central:discrete} (2)? We will investigate this problem in future work.
\end{remark}

\begin{remark}
According to Theorems \ref{Thm:Asym:Binary}--\ref{Thm:ConsistencyDisc},
an approximate $100(1-\alpha)\%$ confidence interval for $\theta_i - \theta_j$ is
$\hat{\theta}_i - \hat{\theta}_j \pm Z_{1-\alpha/2}(1/\hat{v}_{ii}+1/\hat{v}_{jj})^{1/2}$, where $\hat{v}_{ii}$ and
$\hat{v}_{jj}$ are the natural estimates of $v_{ii}$ and $v_{jj}$ by replacing all $\theta_1, \ldots, \theta_n$
by their MLEs, and $Z_{\beta}$ denotes the $100\beta$ percentile point of the standard normal distribution.
To test whether $\theta_i=\theta_j$ at level $\alpha$, the hypothesis can be rejected if
$|\hat{\theta}_i-\hat{\theta}_j|> Z_{1-\alpha/2}(1/\hat{v}_{ii}+1/\hat{v}_{jj})^{1/2}$.
The confidence intervals for contrasts and the hypothesis test for the equality of two parameters
can be generalized to multiple parameters.
For example, one can use the test statistic
\[
(\hat{\theta}_1 -\hat{\theta}_2, \hat{\theta}_2 -\hat{\theta}_3, \hat{\theta}_3-\hat{\theta}_4)
\begin{pmatrix}
\frac{1}{\hat{v}_{11}}+\frac{1}{\hat{v}_{22}} & \frac{-1}{\hat{v}_{22}} & 0 \\
\frac{-1}{\hat{v}_{22}} & \frac{1}{\hat{v}_{22}}+\frac{1}{\hat{v}_{33}} & \frac{-1}{\hat{v}_{33}} \\
0 & \frac{-1}{\hat{v}_{33}}  & \frac{1}{\hat{v}_{33}}+\frac{1}{\hat{v}_{44}}
\end{pmatrix}^{-1}
\begin{pmatrix}
\hat{\theta}_1 -\hat{\theta}_2 \\ \hat{\theta}_2 -\hat{\theta}_3 \\ \hat{\theta}_3-\hat{\theta}_4
\end{pmatrix}
\]
to test whether $\theta_1=\theta_2=\theta_3=\theta_4$, which asymptotically follows the chi-square
distribution with the degree of freedom $3$.
\end{remark}

\section{Simulations}

In this section, we will evaluate the asymptotic results for maximum entropy models on weighted graphs
with continuous and infinite discrete weights through numerical simulations.
The simulation results for finite discrete weights are similar to the binary case, which has been shown in Yan and Xu (2013), so we do not repeat it here.
Firstly, we study the consistency of the estimation. We plot the estimated $\hat{\theta}$ vs $\theta$ to evaluate the
accuracy. Secondly, by Theorems \ref{Thm:Asym:Cont} and \ref{Thm:ConsistencyDisc}, $\hat{v}_{ii}^{1/2}(\hat{\theta}_i-\theta_i)$
and $(\hat{\theta}_i+\hat{\theta}_j-\theta_i-\theta_j)/(1/\hat{v}_{ii}+1/v_{ii})^{1/2}$
are asymptotically normally distributed, where $\hat{v}_{ii}$ is the estimator of $v_{ii}$
by replacing $\theta_i$ with $\hat{\theta}_i$.
The quantile-quantile (QQ) plots of $\hat{v}_{ii}^{1/2}(\hat{\theta}_i-\theta_i)$ are shown.
We also report the $95\%$ coverage
probabilities for certain $\theta_i-\theta_j$, as well as the
probabilities that the maximum likelihood estimator does not exist in the case of discrete weights. The parameter settings in simulation studies are listed as follows.
For continuous weights, let $\theta_i=\widetilde{M} + i\widetilde{M}^2/n,~i=1,\ldots, n$
such that $L_n\approx \widetilde{M}$, $M_n\approx \widetilde{M}^2$ and $M_n/L_n\approx \widetilde{M}$;
for discrete weights, let $\theta_i=0.1+ i\widetilde{M}/n,~i=1, \ldots, n$ such that $L_n\approx 0.1$, $M_n\approx \widetilde{M}$, and
$e^{M_n}/L_n\approx 10e^{\widetilde{M}}$. Here, we suppress the subscript $n$ of $\widetilde{M}$
in order to conveniently display the notations in the figures.
A variety of $\widetilde{M}$ are chosen: $\widetilde{M}=1$, $\log(n)$, $n^{1/2}$, $n$ for continuous weights;
$\widetilde{M}=0$, $\log(\log n)$, $(\log n)^{1/2}$, $\log n$ for discrete weights.

The plots of $\hat{\theta}_i$ vs $\theta_i$ are shown in Figure 1.
We used $\widetilde{M}=1$ in this figure for the case of discrete weights instead of $\widetilde{M}=0$
in order to make $\theta_i$ vary (When $\widetilde{M}=0$, all the $\theta_i$, $1\le i\le n$ equal to $0.1$).
The red lines correspond to the case that $\theta=\hat{\theta}$.
For each sub-figure, the first and second rows represent $n=100$ and $n=200$, respectively.
The first, second and third columns represent $\widetilde{M} =1, \log (n), n^{1/2}$ for continuous weights and
$\widetilde{M}=1, \log\log(n), (\log n)^{1/2}$ for discrete weights, respectively.
From this figure, we can see that as $n$ increases, the estimators become more close to the true parameters.
As $\widetilde{M}$ increases, $\max_i|\hat{\theta}_i - \theta_i|$ becomes much larger, indicting that
controlling the increasing rate of $M_n$ (or decreasing rate of $L_n$) is necessary.
For continuous weights, when $\widetilde{M}=n^{1/2}$, $\max_i|\hat{\theta}_i - \theta_i|$ are very large, exceeding $30$;
for discrete weights, when $\widetilde{M}=(\log (n))^{1/2}$, the points of $\hat{\theta}$ vs $\theta$ diverge, indicating that $\hat{\theta}_i$ may not be the consistent estimate of $\theta_i$ in this case. Therefore, the conditions to guarantee the consistency results
in Theorems \ref{Thm:Asym:Cont} and \ref{Thm:ConsistencyDisc} seem to be
reasonable.

The QQ plots in Figures \ref{figure-continuous-qq} and \ref{figure-discrete-qq}
are based on $5,000$ repetitions for each scenario.
The horizontal and vertical axes are the empirical and theoretical quantiles, respectively.
The red lines correspond to $y=x$.
The coverage frequencies are reported in Table \ref{Table:continuous:dis}.
When $\widetilde{M}=\log n$, the MLEs for the case of discrete weights do not exist with $100\%$ frequencies.
Therefore, the QQ plots for this case are not available.
In Figure \ref{figure-continuous-qq}, when $\widetilde{M}=1, \log n, n^{1/2}$, the sample quantiles coincide with the theoretical ones
very well (The plot of the case of $\widetilde{M}=1$ is similar to that of $\widetilde{M}=\log n$ and is not shown here).
On the other hand, when $\widetilde{M}=n$, the sample quantile of $\hat{v}_{11}^{1/2}(\hat{\theta}_1 -\theta_1)$
evidently deviates from the theoretical one. In this case, the estimated variances $\hat{v}_{ii}$ of $d_i$
are very small, approaching to zero. For example, when $n=200$ and $\widetilde{M}=n$,
the estimated $\hat{v}_{ii}$ is in the magnitude of $10^{-6}\sim 10^{-8}$, where
the central limit theorem cannot be expected according to the classical large sample theory.
In Figure \ref{figure-discrete-qq}, the approximation of asymptotic normality
is good when $\widetilde{M}=0$ and $\log (\log n)$; while there are notable derivations for $\hat{v}_{nn}^{1/2}(\hat{\theta}_n-\theta_n)$
when  $\widetilde{M}=(\log n)^{1/2}$.

In both cases of continuous and discrete weights of Table \ref{Table:continuous:dis},
the length of estimated confidence intervals increases as $\widetilde{M}$ becomes larger when $n$ is fixed.
In the case of continuous weights, when $\widetilde{M}=1, \log(n)$, the length of estimated
confidence intervals decreases as $n$ increases; but when $\widetilde{M}=n^{1/2}$ and $n$, it instead becomes larger.
This is because $v_{ii}$ (between $n(2\widetilde{M})^{-2}$ and $n/(4\widetilde{M}^4)$) goes to zero as $n$ increases
when $\widetilde{M}\ge n^{1/2}$, leading to a larger confidence interval.
In particular, when $\widetilde{M}=n$, some of them exceed $10000$, indicating an extremely inaccurate estimate, although
the corresponding coverage probabilities are close to $95\%$.
In the case of discrete weights, when $\widetilde{M}=0, \log(\log n)$ and $n\ge 100$,
the coverage frequencies are close to the nominal level; when $\widetilde{M}=(\log n)^{1/2}$, the coverage frequencies
of pair $(n-1, n)$ are higher than the nominal level; when $\widetilde{M}=\log n$ that greatly exceeds the condition of Theorem \ref{Thm:ConsistencyDisc},
the MLE almost does not exist.
These phenomena further suggest that controlling increasing rate of $M_n$ or decreasing rate of $L_n$
in Theorems \ref{Thm:Asym:Cont} and \ref{Thm:ConsistencyDisc} is necessary.

\section{Summary and discussion}
Investigating the asymptotic theories for the network models are open and challenging problems,
especially when the number of parameters increases with the size of network.
One reason is that network data are not a standard type of data. In a traditional statistical framework,
the number of parameters is fixed and the number of samples goes to infinity.
In the asymptotic scenario considered in this paper, the sample is only one realization of a random graph
and the number of parameters is identical to that of vertices.
However, 
%%To the best of our knowledge, asymptotic properties
%%of 
the MLE in some simple undirected models with
the degree sequence as the exclusively natural sufficient statistics (i.e., the maximum entropy models)
%% among exponential random graph models when the number of parameters increases with the size of network, 
have been derived.
As the number of parameters goes to infinity, we obtain the asymptotic normality of the MLE in the
maximum entropy models for a class of weighted edges, the proofs of which are in help of
the approximated inverses of the Fisher information matrix. We expect that the methods of our proofs
can be applied to other high-dimensional cases in which the Fisher information matrix are nonnegative and
diagonally dominant or other similar cases. For example, Perry and Wolfe (2012)
introduced a family of null models for network data in which the entries of the upper triangle matrix of $\mathbf{a}$
are assumed independent Bernoulli random variables with success probabilities
$\exp[ \theta_i + \theta_j + \varepsilon(\theta_i, \theta_j)]$ for $1\le i< j \le n$,
where $\varepsilon(\theta_i, \theta_j)$ are smooth functions on parameters $\theta_i$ and $\theta_j$.
By making some assumptions of the second derivative of $\varepsilon$, the Fisher information matrix of the parameters
also shares the similar properties like those in the maximum entropy models.

Finally, we shed some light on why the consistency and asymptotic normality of the MLE can be achieved in the maximum entropy models,
even though the dimension of parameters increases with the size of network and the sample is only one realization of
a random graph. First, in an undirected random graph,
it lurks with $n(n-1)/2$ random variables, which are higher order than the number of parameters. Second,
the Fisher information of each parameter are combinations of $n-1$ variances of $n-1$ random variables.
Under some conditions, it goes to infinity as $n$ increases. Third, the assumption of independently edges avoid the
degeneracy problem, unlike Markov dependent exponential random graphs (Frank and Strauss (1986)).
The model degeneracy problems of the exponential random graphs have received wide attention (e.g., Strauss, 1986; Snijders, 2002;
Handcock, 2003; Hunter and Handcock, 2006; Chatterjee and Diaconis, 2011; Schweinberger, 2011).
%%Since the degree sequence is less than $n$, the probability distribution is stable according to Schweinberger (2011).
Moreover, considering the case that the number of parameters is fixed,
Shalizi and Rinaldo (2013) demonstrated that exponential random graph models are projective in the sense of that
the same parameters can be used for the full network and for any of its subnetworks simultaneously,
essentially only for those models with the assumption of dyadic independence, under which the consistency of the MLE
is available.

\section*{Appendix A}
For fixed $r$, the central limit theorem for the vector $(d_1, \ldots, d_r)$ can be easily derived
by noting that $d_1, \ldots, d_r$ are asymptotically independent.
In view of that $d_i$ is the sufficient statistic on $\theta_i$, $\hat{\theta}_i$ may be approximately represented as a function of $d_i$.
If this can be done, then the asymptotic distribution for the MLE may follow.
In order to establish the relationship between $\hat{\theta}_i$ and $d_i$, we will approximate
the inverse of a class of matrices. We say an $n\times n$ matrix $V_n=(v_{ij})$ belongs to a matrix class $\mathcal{L}_n(m, M)$ if
$V_n$ is a symmetric nonnegative matrix satisfying
\[
v_{ii}=\sum_{j=1;j\neq i}^n v_{ij};~~M\ge v_{ij}=v_{ji}\ge m >0,~~i\neq j.
\]
Yan and Xu (2013) have proposed to use
$\bar{S}_n=\mathrm{diag}(1/v_{11}, \ldots, 1/v_{nn})+ v_{..}^{-1}\mathbf{1}_n \mathbf{1}_n^\prime$
to approximate $V^{-1}_n$,  where
$\mathbf{1}_n$ is a vector of $n$ entries whose values are all of $1$ and $v_{..}=\sum_{i=1}^n v_{ii}$,
and obtained an upper bound on the approximate errors.
Here we use a simpler matrix $S_n=\mathrm{diag}(1/v_{11}, \ldots, 1/v_{nn})$ to approximate $V^{-1}_n$.
Let $\|A\|=\max_{i,j} |a_{ij}|$ for a general matrix $A=(a_{ij})$.
It is clear that $\|A+B\|\le \|A\|+\|B\|$ for two matrices $A$ and $B$, and
$\|\bar{S}_n-S_n \|=v_{..}^{-1} \le (mn(n-1))^{-1}$.
By Proposition A1 in Yan and Xu (2013), we have:
\begin{proposition}\label{Prop:Matrix:Inverse}
If $V_n\in \mathcal{L}_n(m, M)$, then for $n\ge 3$, the following holds:
\begin{eqnarray*} %%\label{O-upperbound}
||W_n:=V^{-1}_n - S_n || & \le & \|V^{-1}_n - \bar{S}_n \| + \|\bar{S}_n-S_n \| \\
& \le & \frac{ M(nM+(n-2)m)}{2m^3(n-2)(n-1)^2}+\frac{1}{2m(n-1)^2}+\frac{1}{mn(n-1)}.
%%% \le \frac{ M(nM+(n-2)m)}{2(n-2)m^3(n-1)^2}+\frac{1}{2v_{\cdot\cdot}}
\end{eqnarray*}
\end{proposition}

Note that $d_i=\sum_{j\neq i}a_{ij}$ are sums of $n-1$
independent multinomial random variables. By the central limit theorem for the bounded case (Lo\`{e}ve, 1977, p. 289),
$v_{ii}^{-1/2} \{d_i - \E(d_i)\}$
is asymptotically standard normal if $v_{ii}$ diverges.
Following Hillar and Wibisono (2013), we assume that $\max_i |\theta_i|$ is bounded by a constant in this appendix.
For convenience, we assume that $\max_{i}|\theta_i| \le L/2$ with $L$ a fixed constant.
Thus, $\max_{i,j}|\theta_i + \theta_j|\le L$.
Since
\[
e^{2k(\theta_i+\theta_j)} \le e^{(k+(k-1))(\theta_i+\theta_j) + L },~~1\le k \le q-1,
\]
we have
\[
\sum_{k=0}^{q-1} e^{2k(\theta_i+\theta_j)} \le \sum_{0\le k\neq l \le q-1} e^{(k+l)(\theta_i+\theta_j)}e^{L}.
\]
Therefore,
\begin{eqnarray}
\nonumber
\frac{\frac{1}{2}\sum_{k\neq l}e^{(k+l)(\theta_i+\theta_j)} }{ (\sum_{a=0}^{q-1} e^{a(\theta_i+\theta_j)})^2 }
& = & \frac{\frac{1}{2}\sum_{k\neq l}e^{(k+l)(\theta_i+\theta_j)} }
{ \sum_{k\neq l} e^{(k+l)(\theta_i+\theta_j)} +\sum_{k=0}^{q-1}e^{2k(\theta_i+\theta_j)} }
\\
\label{inequality-binary-vij}
& \ge & \frac{\sum_{k\neq l}e^{(k+l)(\theta_i+\theta_j)} }
{ 2(1+e^{L}) \sum_{k\neq l} e^{(k+l)(\theta_i+\theta_j)}  }\ge \frac{1}{2(1+e^{L})}.
\end{eqnarray}
Recall the definition of $v_{ii}$ in Theorem \ref{Thm:Asym:Binary}. It shows that $v_{ii}\ge (n-1)/(2(1+e^{L}))$. If $L$ is a constant,
then $v_{ii}\to \infty$ for all $i$ as $n\to\infty$.
If $r$ is a fixed constant, one may replace the statistics $d_1, \ldots, d_r$ by the independent random variables
$\tilde{d}_i=d_{i, r+1} + \ldots + d_{in}$, $i=1,\ldots,r$ when considering the asymptotic behaviors of $d_1, \ldots, d_r$.
Therefore, we have the following proposition.

\begin{proposition}\label{pro:central:binary}
Assume that $\max_i |\theta_i|\le L$ with $L$ a constant. Then as
$n\to\infty$: \\
(1)For any fixed $r\ge 1$,  the components of $(d_1 - \E (d_1), \ldots, d_r - \E (d_r) )$ are
asymptotically independent and normally distributed with variances $v_{11}, \ldots, v_{rr}$,
respectively. \\
(2)More generally, $\sum_{i=1}^n c_i(d_i-\E(d_i))/\sqrt{v_{ii}}$ is asymptotically normally distributed with mean zero
and variance $\sum_{i=1}^\infty c_i^2$ whenever $c_1, c_2, \ldots$ are fixed constants and the latter sum is finite.
\end{proposition}

Part (2) follows from part (1) and the fact that
\[
\lim_{r\to\infty} \limsup_{t\to\infty}Var( \sum_{k=r+1}^n c_i \frac{ d_i - \E (d_i) }{\sqrt{v_{ii}}})=0
\]
by Theorem 4.2 of Billingsley (1968). To prove the above equation, it suffices to show that the eigenvalues of
the covariance matrix of $(d_i - \E (d_i))/v_{ii}^{1/2}$, $i=r+1, \ldots, n$ are bounded by 2 (for all $r<n$).
This comes from the well-known Perron-Frobenius theory:
if $A$ is a symmetric positive definite matrix with diagonal elements equaling to $1$, and with negative off-diagonal elements,
then its largest eigenvalue is less than $2$. We will only use part (1) to prove Theorem \ref{Thm:Asym:Binary}.

Before proving Theorem \ref{Thm:Asym:Binary}, we show three lemmas below.
By direct calculations,
\[
v_{ij}=
\frac{\sum_{0\le k< l\le q-1}(k-l)^2e^{(k+l)(\theta_i+\theta_j)} }{ (\sum_{a=0}^{q-1} e^{a(\theta_i+\theta_j)})^2 },~~i,j=1, \ldots, n;i\neq j.
\]
and $v_{ii}=\sum_{j\neq i}v_{ij}$.
On the other hand, it is easy to see that
\begin{equation}\label{ineq:vij-binary:upper}
\frac{\frac{1}{2}\sum_{k\neq l}(k-l)^2e^{(k+l)(\alpha_i+\alpha_j)} }{ (\sum_{a=0}^{q-1} e^{a(\alpha_i+\alpha_j)})^2 }
\le \frac{1}{2}\max_{k\neq l}(k-l)^2 \le \frac{q^2}{2}.
\end{equation}
In view of inequality \eqref{inequality-binary-vij},
if $\max_{i,j}|\theta_i + \theta_j| \le L$ with $L$ a constant, then $V_n\in L_n(m, M)$ with $m$ and $M$ constants.
Applying Proposition \ref{Prop:Matrix:Inverse}, we have
\begin{lemma}\label{lemma:binary:matrix1}
Assume that $\max_{i}|\theta_i| \le L/2$ with $L$ a fixed constant. If $n$ is large enough, then
\begin{equation}
\| V_n^{-1} - S_n \| \le c_1(n-1)^{-2},
\end{equation}
where $c_1$ is a constant only depending on $L$.
\end{lemma}

\begin{lemma}\label{lemma:binary:2}
Assume that $\max_{i}|\theta_i| \le L/2$ with $L$ a fixed constant. Let $U_n=\mbox{cov}[W_n\{\mathbf{d}-\E(\mathbf{d})\}]$. Then
\begin{equation}
\|U_n\|\le \|V_n^{-1}-S_n\| + c_2(n-1)^{-2},
\end{equation}
where $c_2$ is a constant only depending on $L$.
\end{lemma}
\begin{proof}
Note that
\begin{equation*}
U_n=W_nV_nW^T_n=(V_n^{-1}-S_n)-S_n (I_n-V_nS_n),
\end{equation*}
and
\begin{equation*}
\{S_n(I_n-V_nS_n)\}_{i,j}=\frac{(\delta_{ij}-1)v_{ij}}{v_{ii}v_{jj}}.
\end{equation*}
By  \eqref{inequality-binary-vij} and \eqref{ineq:vij-binary:upper},
\begin{equation*}
|\{S_n(I_n-V_nS_n)\}_{ij}|\le c_2(n-1)^{-2},
\end{equation*}
where $c_2$ is a constant.
Thus,
\begin{eqnarray*}
\|U_n\|  \le  \|V_n^{-1}-S_n\|+\|S_n(I_n-V_nS_n)\|
\le
\|V_n^{-1}-S_n\|+c_2(n-1)^{-2}.
\end{eqnarray*}
\end{proof}

In order to prove the below lemma, we need one theorem due to Hillar and Wibisono (2013).

\begin{theorem}\label{Thm:Asym:Binary-con}
Assume that $\max_i |\theta_i| \le L/2$ with $L>0$ a fixed constant.
Then for sufficiently large $n$, with probability at least $1-2/n$, the MLE $\bs{\widehat{\theta}}$ exists and satisfies
\[
\max_i |\hat{\theta}_i - \theta_i| \le c_3\sqrt{\frac{\log n}{n} }
\]
where $c_3$ is a constant that only depends on $L$.
\end{theorem}

\begin{lemma}\label{lemma:binary:3}
If $\max_i |\theta_i| \le L/2$ with $L>0$ a fixed constant, then
for $i=1, \ldots, r$ with $r$ a fixed constant,
\begin{equation*}
\theta_i - \hat{\theta}_i = [V_{n}^{-1}\{\mathbf{d} - \E(\mathbf{d}) \}]_i + o_p(n^{-1/2}).
\end{equation*}
\end{lemma}

\begin{proof}
Let $E_n$ be the event that the MLE $\bs{\widehat{\theta}}$ exists and $F_n$ be the event that
$\lambda_n:=\max_i |\hat{\theta}_i - \theta_i| \le c(\log n)^{1/2}/n^{1/2}$.
Derivations in what follows are on the event $E_n\bigcap F_n$.
Let
\[
\mu(t)= \sum_{a=0}^{r-1} \frac{ae^{at}}{\sum_{k=0}^{r-1} e^{kt}}.
\]
It is easy to verify that
\[
\mu'(t)=\frac{\sum_{0\le k< l\le r-1}(k-l)^2e^{(k+l)t} }{ (\sum_{a=0}^{r-1} e^{at})^2 }
\]
and
\[
\mu''(t)=\left[ \frac{\frac{1}{2} \sum_{k\neq l, a}(k-l)^2(k+l-2a)e^{(k+l+a)t}
}{(\sum_{a=0}^{r-1} e^{at})^3 } \right],
\]
such that
\begin{equation}\label{ineq:mu-second}
|\mu''(t)|\le (r-1)^3.
\end{equation}
Applying Taylor's expansions to $\mu(\hat{\theta}_i + \hat{\theta}_j)$ at the point $\theta_i + \theta_j$,
for $i=1, \ldots, n$, we have
\begin{eqnarray*}
d_i - E(d_i) & = & \sum_{j\neq i} (\mu(\hat{\theta}_i + \hat{\theta}_j) - \mu( \theta_i + \theta_j) ) \\
&=& \sum_{j\neq i} [\mu'(\theta_i+\theta_j)(\mu(\hat{\theta}_i + \hat{\theta}_j) - \mu( \theta_i + \theta_j) )] + h_i,
\end{eqnarray*}
where $h_i=\frac{1}{2}\sum_{j\neq i} \mu''( \hat{\gamma}_{ij} )[((\hat{\theta}_i + \hat{\theta}_j) - ( \theta_i + \theta_j) )]^2$,
and $\hat{\gamma}_{ij} = t_{ij}( \theta_i + \theta_j)
+ (1-t_{ij})(\hat{\theta}_i + \hat{\theta}_j)$, $0<t_{ij}<1$.
Writing the above expressions into a matrix, it yields,
\[
\mathbf{d} - \E (\mathbf{d}) = V_n(\bs{\widehat{\theta}} - \bs{\theta}) + \mathbf{h},
\]
or equivalently,
\begin{equation}\label{binary:represent}
\bs{\widehat{\theta}} - \bs{\theta}  =  V^{-1}_n( \mathbf{d} - \E \mathbf{d}) + V^{-1}_n \mathbf{h} ,
\end{equation}
where $\mathbf{h}=(h_1, \ldots, h_n)^T$.
By \eqref{ineq:mu-second}, $|h_i|\le \frac{1}{2}(n-1)(r-1)^3\hat{\eta}_{ij}^2$, where $\hat{\eta}_{ij}=(\hat{\theta}_i + \hat{\theta}_j) - ( \theta_i + \theta_j)$.
Therefore, by Lemma \ref{lemma:binary:matrix1},
\begin{eqnarray*}
|(V^{-1}_n \mathbf{h})_i| & = & |(S_n \mathbf{h})_i| + |(W_n\mathbf{h})_i|
\\
& \le & \max_i \frac{|h_i|}{v_{ii}} + \|W\|\sum_i|h_i|
= O(\frac{\log n}{n}).
\end{eqnarray*}

By Theorem \ref{Thm:Asym:Binary-con}, $P(E_n\bigcap F_n)\to 1$ as $n\to\infty$. It shows
$(V^{-1} \mathbf{h})_i = o_p(n^{-1/2})$ for $i=1,\ldots, r$ with $r$ a fixed constant.
Consequently, we have $\theta_i - \hat{\theta}_i = [V_{n}^{-1}\{\mathbf{d} - \E(\mathbf{d}) \}]_i + o_p(n^{-1/2})$ for $i=1, \ldots, r$.
\end{proof}

\begin{proof}[Proof of Theorem~\ref{Thm:Asym:Binary}]
By \eqref{binary:represent},
\begin{equation*}
(\boldsymbol{\theta} - \boldsymbol{\hat{\theta}} )_i= [S_n\{\mathbf{d}-\E(\mathbf{d})\}]_i+ [W_n\{\mathbf{d} - \E(\mathbf{d})\}]_i + (V_n^{-1}\mathbf{h})_i.
\end{equation*}
By Lemmas \ref{lemma:binary:2} and \ref{lemma:binary:3}, if $\max_i |\theta_i| \le L/2$ with $L>0$ a fixed constant, then
\[
(\boldsymbol{\theta} - \boldsymbol{\hat{\theta}} )_i = \frac{ d_i - \E(d_i)}{v_{ii}} + o_p(n^{-1/2}).
\]
Theorem \ref{Thm:Asym:Binary} follows directly from Proposition \ref{pro:central:binary}, part (1).
\end{proof}

\section*{Appendix B}
In the case of continuous weights, the elements of $V_n$ are
\[
v_{ii}=\sum_{j\neq i}\frac{1}{(\theta_i + \theta_j)^2},~~i=1,\ldots, n;~~
v_{ij}=\frac{1}{(\theta_i+\theta_j)^2},~~i\neq j.
\]
Note that $V_n$ is also the covariace matrix of $\mathbf{d}$.
Recall that $L_n:=\min_{i\neq j} (\theta_i+\theta_j)$ and $M_n:=\max_{i\neq j}(\theta_i+\theta_j)$.
Therefore,
\begin{equation}\label{ineq:vij:conti}
\frac{1}{M_n^2} \le v_{ij} \le \frac{1}{L_n^2},~~i\neq j,~~~ \frac{(n-1)}{M_n^2}\le v_{ii} \le \frac{(n-1)}{L_n^2},
~~ i=1, \ldots, n.
\end{equation}
Applying Proposition \ref{Prop:Matrix:Inverse}, we have:
\begin{lemma}\label{lemma:matrix-continuous}
If $n$ is large enough, then
\begin{equation}
\| V_n^{-1} - S_n \| \le \frac{cM_n^2}{L_n^3 (n-1)^2 },
\end{equation}
where $c$ is a constant.
\end{lemma}

Note that $d_i=\sum_{j\neq i}a_{ij}$ are sums of $n-1$
independent exponential random variables.
It is easy to show that the third moment of the exponential random variable with rate parameter $\lambda$ is $6\lambda^{-3}$.
Then we have
\[
\frac{ \sum_{j\neq i} \E (a_{ij}^3) }{ v_{ii}^{3/2} } = \frac{ 6\sum_{j\neq i} (\theta_i + \theta_j)^{-1} }{ v_{ii}^{1/2} } \le
\frac{ 6M_n/L_n}{(n-1)^{1/2}}.
\]
If $M_n/L_n=o(n^{1/2})$, then the above expression goes to zero. This shows that the condition for
the Lyapunov's central limit theorem, holds.
Therefore, $v_{ii}^{-1/2} \{d_i - \E(d_i)\}$
is asymptotically standard normal if $M_n/L_n=o(n^{1/2})$.
Similar to Proposition \ref{pro:central:binary}, we have the proposition below.

\begin{proposition}\label{pro:central:continuous}
If $M_n/L_n=o( n^{1/2} )$, then as
$n\to\infty$: \\
(1)For any fixed $r\ge 1$,  the components of $(d_1 - \E (d_1), \ldots, d_r - \E (d_r))$ are
asymptotically independent and normally distributed with variances $v_{11}, \ldots, v_{rr}$,
respectively. \\
(2)More generally, $\sum_{i=1}^n c_i(d_i-\E(d_i))/\sqrt{v_{ii}}$ is asymptotically normally distributed with mean zero
and variance $\sum_{i=1}^\infty c_i^2$ whenever $c_1, c_2, \ldots$ are fixed constants and the latter sum is finite.
\end{proposition}

Before proving Theorem \ref{Thm:Asym:Cont}, we show the following two lemmas.
The proof of Lemma \ref{con-central1-lemma1} is similar to that of Lemma \ref{lemma:binary:2} and we omit it.

\begin{lemma}\label{con-central1-lemma1}
Let $U_n=\mbox{cov}[W_n \{\mathbf{d}-\E(\mathbf{d})\}]$. Then
\begin{equation}
||U_n||\le ||V_n^{-1}-S_n||+\frac{M_n^2}{L_n^2(n-1)^2}.
\end{equation}
\end{lemma}
\iffalse
\begin{proof}
Note that
\begin{equation*}
U_n=W_nV_nW^T_n=(V_n^{-1}-S_n)-S_n (I_n-V_nS_n),
\end{equation*}
and
\begin{equation*}
\{S_n(I_n-V_nS_n)\}_{i,j}=\frac{(\delta_{ij}-1)v_{ij}}{v_{ii}v_{jj}}.
\end{equation*}
By  \eqref{ineq:vij:conti},
\begin{equation*}
|\{S_n(I_n-V_nS_n)\}_{ij}|\le \frac{M_n^2}{L_n^2(n-1)^2},
\end{equation*}
Thus,
\begin{eqnarray*}
\|U_n\|  \le  \|V_n^{-1}-S_n\|+\|S_n(I_n-V_nS_n)\|
\le
\|V_n^{-1}-S_n\|+\frac{M_n^2}{L_n^2(n-1)^2}.
\end{eqnarray*}
\end{proof}
\fi

In order to prove the lemma below, we need one theorem due to Hillar and Wibisono (2013).
\begin{theorem}\label{Thm:Asym:Cont:con}
Let $k > 1$ be fixed. Then for sufficiently large $n$, with probability at least $1-3n^{-(k-1)}$,
the MLE $\boldsymbol{\widehat{\theta}}$ exists and satisfies
\begin{equation*}
\|\boldsymbol{\widehat \theta} - \boldsymbol{\theta}\|_\infty \leq \frac{150 M_n^2}{L_n} \sqrt{\frac{k \log n}{n}}.
\end{equation*}
%%%If $M_n^2/L_n=o\{n^{1/2}/(\log n)^{1/2}\}$, then $\boldsymbol{\widehat{\theta}}$ is uniformly consistent as $n$ goes to infinity. \\
\end{theorem}

\begin{lemma}\label{concentral-lemma2}
If $M_n/L_n=o\{ n^{1/16}/(\log n)^{1/8} \}$, then for $i=1, \ldots, r$ with a fixed constant $r$,
\begin{equation}
\theta_i - \hat{\theta}_i = [V_{n}^{-1}\{\mathbf{d} - \E(\mathbf{d}) \}]_i + o_p(n^{-1/2}).
\end{equation}
\end{lemma}

\begin{proof}
By Theorem \ref{Thm:Asym:Cont:con} (b), if $M_n/L_n=o\{ n^{1/16}/(\log n)^{1/8} \}$, then
\begin{equation}\label{lambda-n-bound}
\lambda_n=\max_{1\le i\le n} |\hat{\theta}_i-\theta_i|=O_p\{ \frac{M_n^2}{L_n^2}\sqrt{\frac{\log n}{n} } \}.
\end{equation}
For $i=1,\ldots, n$, direct calculations give
\begin{eqnarray*}
d_i - \E (d_i) & = & \sum_{j\neq i} \left ( \frac{1}{
\hat{\theta}_i + \hat{\theta}_j } - \frac{1}{\theta_i + \theta_j} \right )
= \sum_{j\neq i} \frac{\theta_i-\hat{\theta}_i + \theta_j-\hat{\theta}_j }{(\hat{\theta}_i + \hat{\theta}_j)(\theta_i + \theta_j) }
\\
& = & \sum_{j\neq i}\left [ \frac{\theta_i-\hat{\theta}_i + \theta_j-\hat{\theta}_j }{(\theta_i + \theta_j)^2 }\left(
\frac{\theta_i + \theta_j}{\hat{\theta}_i + \hat{\theta}_j} - 1 \right ) + \frac{\theta_i-\hat{\theta}_i + \theta_j-\hat{\theta}_j }{(\theta_i + \theta_j)^2 }
\right ]
\\
& = & \sum_{j\neq i}\left [ \frac{ (\theta_i-\hat{\theta}_i + \theta_j-\hat{\theta}_j)^2 }{(\theta_i + \theta_j)^2(\hat{\theta}_i + \hat{\theta}_j) } + \frac{\theta_i-\hat{\theta}_i + \theta_j-\hat{\theta}_j }{(\theta_i + \theta_j)^2 }
\right ].
\end{eqnarray*}
Writing the above expression for $i=1, \ldots, n$ into the form of a matrix, it yields
\[
\mathbf{d} - \E(\mathbf{d}) = V_n (\boldsymbol{\theta} - \boldsymbol{\hat{\theta}} ) + \mathbf{h},
\]
where $\mathbf{h}=(h_1, \ldots, h_n)^\prime$ and
\[
h_i = \sum_{j\neq i} \frac{ (\theta_i-\hat{\theta}_i + \theta_j-\hat{\theta}_j)^2 }{(\theta_i + \theta_j)^2(\hat{\theta}_i + \hat{\theta}_j) }:=\sum_{j\neq i}h_{ij}.
\]
Equivalently,
\begin{equation}\label{represent}
\boldsymbol{\theta}-\boldsymbol{\hat{\theta}} =V_n^{-1}\{\mathbf{d} - \E(\mathbf{d})\} + V_n^{-1}\mathbf{h}.
\end{equation}

In view of \eqref{ineq:vij:conti} and \eqref{lambda-n-bound}, we have
\[
|h_{ij}|\le  \frac{2\lambda_n^2}{ L_n^2(L_n-\lambda_n) };~~~|h_i|\le n\times \frac{2\lambda_n^2}{ L_n^2(L_n-\lambda_n) }.
\]
Note that $(S_n \mathbf{h})_i  =  h_i/v_{ii}$ and $(V_n^{-1} \mathbf{h})_i=(S_n \mathbf{h})_i+(W_n \mathbf{h})_i$. By
direct calculation, we have
\begin{equation*}
|(S_n \mathbf{h})_i|  \le \frac{2\lambda_n^2M_n^2}{L_n^2(L_n-\lambda_n)},
\end{equation*}
and, by Lemma \ref{lemma:matrix-continuous},
\begin{equation*}
|(W_n h)_i| \le  \|W_n\|\times (n\max_i|h_i|) \le  O_p( \frac{M_n^6}{L_n^8}\times \frac{\log n}{n} ).
\end{equation*}
If $M_n/L_n = o\{ n^{1/16}/(\log n)^{1/8} \}$, then $|(V_n^{-1}\mathbf{h})_i|\le
|(S_n\mathbf{h})_i|+|(W_n\mathbf{h})_i|=o_p(n^{-1/2})$. This completes the proof.
\end{proof}

\begin{proof}[Proof of Theorem~\ref{Thm:Asym:Cont}]
By \eqref{represent},
\begin{equation*}
(\boldsymbol{\theta} - \boldsymbol{\hat{\theta}} )_i= [S_n\{\mathbf{d}-\E(\mathbf{d})\}]_i+ [W_n\{\mathbf{d} - \E(\mathbf{d})\}]_i + (V_n^{-1}\mathbf{h})_i.
\end{equation*}
By Lemmas \ref{con-central1-lemma1} and \ref{concentral-lemma2}, if $M_n/L_n = o\{ n^{1/16}/(\log n)^{1/8} \}$, then
\[
\theta_i - \hat{\theta}_i = (d_i - \E(d_i))/v_{ii} + o_p(n^{-1/2}).
\]
Theorem \ref{Thm:Asym:Cont} follows directly from Proposition \ref{pro:central:continuous}, part (1).
\end{proof}

\section*{Appendix C}

Note that $d_i=\sum_{j\neq i}a_{ij}$ is a sum of $n-1$ geometric random variables.
Recall the definitions of $L_n$ and $M_n$, we have
\[
\frac{e^{M_n}}{(e^{M_n} - 1)^2 } \le v_{ij}= \frac{e^{\theta_i+\theta_j}}{(e^{\theta_i+\theta_j}-1)^2}
\le \frac{e^{L_n} }{ (e^{L_n} - 1)^2 },~~1\le i< j\le n,
 \]
 \iffalse
 \[
 \frac{(n-1)e^{M_n}}{(e^{M_n} - 1)^2 } \le v_{ii} \le \frac{(n-1)e^{L_n} }{ (e^{L_n} - 1)^2 },~~ i=1, \ldots, n.
\]
\fi
This shows $V_n\in L_n(m, M)$ with $m$ and $M$ given by the lower and upper bounds in the above inequalities on $v_{ij}$.
By the moment-generating function of the geometric distribution, it is easy to verify that
\[
\E (a_{ij}^3) = \frac{1-p_{ij}}{p_{ij}}+\frac{6(1-p_{ij})}{p_{ij}^2}+\frac{6(1-p_{ij})^2}{p_{ij}^3},
\]
where $p_{ij}=1-e^{-(\theta_i+\theta_j)}$. By simple calculations, we have
\[
\E (a_{ij}^3) = v_{ij}(6 + \frac{ e^{\theta_i+\theta_j}-1}{e^{\theta_i+\theta_j}} + \frac{ 6}{ e^{\theta_i+\theta_j} -1 }).
\]
It follows
\[
\frac{\sum_{j\neq i} \E (a_{ij}^3) }{ v_{ii}^{3/2} } \le \frac{ 7 + 6(e^{L_n} -1 )^{-1} }{v_{ii}^{1/2}}
\le \frac{ [7 + 6(e^{L_n} -1 )^{-1}](e^{M_n}-1) }{ n^{1/2} e^{M_n/2}}.
\]
If $e^{M_n/2}/L_n = o( n^{1/2} )$, then the above expression goes to zero.
This shows that the condition for the Lyapunov's central limit theorem holds.
Therefore, $v_{ii}^{-1/2} \{d_i - \E(d_i)\}$ is asymptotically standard normal under the condition $e^{M_n/2}/L_n = o( n^{1/2} )$.
Similar to Proposition \ref{pro:central:binary}, and Lemmas \ref{lemma:binary:matrix1} and \ref{lemma:binary:2}, we have
the following proposition and lemmas \ref{lemma:matrix-discrete} and \ref{dis-central1-lemma1}.

\begin{proposition}\label{pro:central:discrete}
If $e^{M_n/2}/L_n=o( n^{1/2} )$, then as $n\to\infty$: \\
(1)For any fixed $r\ge 1$,  the components of $(d_1 - \E (d_1), \ldots, d_r - \E (d_r))$ are
asymptotically independent and normally distributed with variances $v_{11}, \ldots, v_{rr}$,
respectively. \\
(2)More generally, $\sum_{i=1}^n c_i(d_i-\E(d_i))/\sqrt{v_{ii}}$ is asymptotically normally distributed with mean zero
and variance $\sum_{i=1}^\infty c_i^2$ whenever $c_1, c_2, \ldots$ are fixed constants and the latter sum is finite.
\end{proposition}

\begin{lemma}\label{lemma:matrix-discrete}
If $n$ is large enough, then
\begin{equation}
\| V_n^{-1} - S_n \| \le \frac{ce^{3M_n} }{L_n^4 (n-1)^2 },
\end{equation}
where $c$ is a constant.
\end{lemma}

\begin{lemma}\label{dis-central1-lemma1}
Let $U_n=\mbox{cov}[W_n\{\mathbf{d}-\E(\mathbf{d})\}]$. Then
\begin{equation}
\|U_n\|\le ||V_n^{-1}-S_n|| + \frac{ e^{L_n}(e^{M_n}-1)^4}{(n-1)^2(e^{L_n}-1)^2e^{2M_n} }.
\end{equation}
\end{lemma}

In order to prove the lemma below, we need one theorem due to Hillar and Wibisono (2013).

\begin{theorem}\label{Thm:ConsistencyDisc-con}
%%(a). If $M_n=o(\log n)$, then as $n$ goes to infinity, with probability approaching one,
%%the MLE exists. If it exists, then it is unique.\\
If $e^{5M_n}/L_n^{1/2}=o\{(n/\log n)^{1/2}\}$, the MLE $\boldsymbol{\widehat{\theta}}$ exists with probability approaching one and
is uniformly consistent in the sense that
\begin{equation*}
\|\boldsymbol{\widehat \theta} - \boldsymbol{\theta}\|_\infty
\leq O_p\left( \frac{(\exp(5M_n)-1)^2}{\exp(5M_n)} \sqrt{\frac{12}{\exp(L_n/2)-1}} \: \sqrt{\frac{2 \log n}{n}} \right )=o_p(1).
\end{equation*}
\end{theorem}

\begin{lemma}\label{dis-central-lemma2}
If $e^{17M_n}/ L_n^3 =o\{n^{1/2}/\log n \}$, then for $i=1, \ldots, r$ with a fixed constant $r$,
\begin{equation}
\theta_i - \hat{\theta}_i =[V_{n}^{-1}\{\mathbf{d} - \E(\mathbf{d})\}]_i+o_p(n^{-1/2}).
\end{equation}
\end{lemma}
\begin{proof}
By Theorem \ref{Thm:ConsistencyDisc-con} (b), if $e^{17M_n}/ L_n^3 =o\{n^{1/2}/\log n \}$, then
\[
\lambda_n=\max_{1\le i\le n} |\hat{\theta}_i-\theta_i|=O_p\Big\{ \frac{(e^{5M_n}-1)^2}{e^{5M_n}}\sqrt{\frac{\log n}{n(e^{L_n/2}-1)}} \Big\}.
\]
Let $\hat{\gamma}_{ij}=\theta_i + \theta_j - \hat{\theta}_i - \hat{\theta}_j$.
By Taylor's expansion, for any $i\neq j$,
\[
\frac{1}{e^{\hat{\theta}_i+\hat{\theta}_j}-1} - \frac{1}{e^{\theta_i+\theta_j}-1}
=
\frac{e^{\theta_i+\theta_j}}{(e^{\theta_i+\theta_j}-1)^2}\hat{\gamma}_{ij} + h_{ij},
\]
where
\[
h_{ij}= -\frac{e^{\theta_i+\theta_j+\alpha_{ij}\hat{\gamma}_{ij}}(1+e^{\theta_i+\theta_j+\alpha_{ij}\hat{\gamma}_{ij}}) }{(e^{\theta_i+\theta_j+\alpha_{ij}\hat{\gamma}_{ij}}-1)^3}\hat{\gamma}_{ij}^2,
\]
and $0 < \alpha_{ij}< 1$. It is easy to verify that
\begin{equation*}
\mathbf{d}  - \E(\mathbf{d}) = V_n( \boldsymbol{\theta} - \boldsymbol{\hat{\theta}} ) + \mathbf{h},
\end{equation*}
where $\mathbf{h}=(h_1, \ldots, h_n)^T$ and $h_i=\sum_{j\neq i}h_{ij}$.
Equivalently,
\begin{equation}\label{represent2}
 \boldsymbol{\theta} - \boldsymbol{\hat{\theta}} = V_n^{-1} \{\mathbf{d} - \E (\mathbf{d})\} + V_n^{-1}\mathbf{h}.
\end{equation}
Since $\theta_i + \theta_j >0 $ and $\lambda_n$ is sufficiently small, we have
\[
|h_{ij}| \le \frac{ e^{2(M_n+\lambda_n)}(1+e^{2(M_n+\lambda_n)})}{[e^{2(L_n-\lambda_n)}-1]^3}\lambda_n^2,~~~~
|h_i|\le \sum_{j\neq i}|h_{ij}|\le (n-1)\frac{ e^{2(M_n+\lambda_n)}(1+e^{2(M_n+\lambda_n)})}{[e^{2(L_n-\lambda_n)}-1]^3}\lambda_n^2.\]
 Note that
$(S_n \mathbf{h})_i  =  h_i/v_{ii}$ and $(V_n^{-1} \mathbf{h})_i=(S_n \mathbf{h})_i+(W_n \mathbf{h})_i$. By
direct calculation, we have
\begin{equation*}
|(S_n \mathbf{h})_i|  \le \frac{(e^{M_n}-1)^2}{e^{M_n}}\times \frac{ e^{2(M_n+\lambda_n)}(1+e^{2(M_n+\lambda_n)})}{[e^{2(L_n-\lambda_n)}-1]^3} \lambda_n^2
= O\{ \frac{e^{15M_n}\log n}{nL_n^3}  \},
\end{equation*}
and, by Lemma \ref{lemma:matrix-discrete},
\begin{equation*}
|(W_n h)_i| \le  \|W_n\|\times (n\max_i|h_i|)  \le O\{ \frac{e^{17M_n}\log n}{n L_n^3 }
\}.
\end{equation*}
If $e^{17M_n}/ L_n^3 =o\{n^{1/2}/\log n \}$, then $|(V_n^{-1}\mathbf{h})_i|\le |(S_n\mathbf{h})_i|+|(W_n\mathbf{h})_i|=o(n^{-1/2})$. This completes the proof.
\end{proof}

\begin{proof}[Proof of Theorem~\ref{Thm:ConsistencyDisc}]
By \eqref{represent2},
\begin{equation*}
(\boldsymbol{\theta} - \boldsymbol{\hat{\theta}} )_i= [S_n\{\mathbf{d} - \E(\mathbf{d})\}]_i+ [W_n\{\mathbf{d}-\E(\mathbf{d})\}]_i + (V_n^{-1}\mathbf{h})_i.
\end{equation*}
By Lemmas \ref{dis-central1-lemma1} and \ref{dis-central-lemma2}, if $e^{17M_n}/ L_n^3 =o\{n^{1/2}/\log n \}$, then
\[
\hat{\theta}_i - \theta_i= (d_i - \E(d_i))/v_{ii} + o_p(n^{-1/2}).
\]
Theorem \ref{Thm:ConsistencyDisc} follows
directly from Proposition \ref{pro:central:discrete}.
\end{proof}

\section*{Acknowledgements}
We are very grateful to two anonymous referees, an associate editor, and the Editor for
their valuable comments that have greatly improved the manuscript. Yan was partially supported by
by the National Natural Science Foundation of China (No. 11341001) 
and Postdoctoral Science Foundation of China.

\clearpage
\newpage

\begin{figure}[h]
\centering
\caption{The plots of $\hat{\theta}_i$ vs $\theta_i$ for $1\le i\le n$ based on $20$ samples. }
\label{figure1}
\subfigure[Continuous weights]{
\includegraphics[ height=4in, width=6in, angle=0]{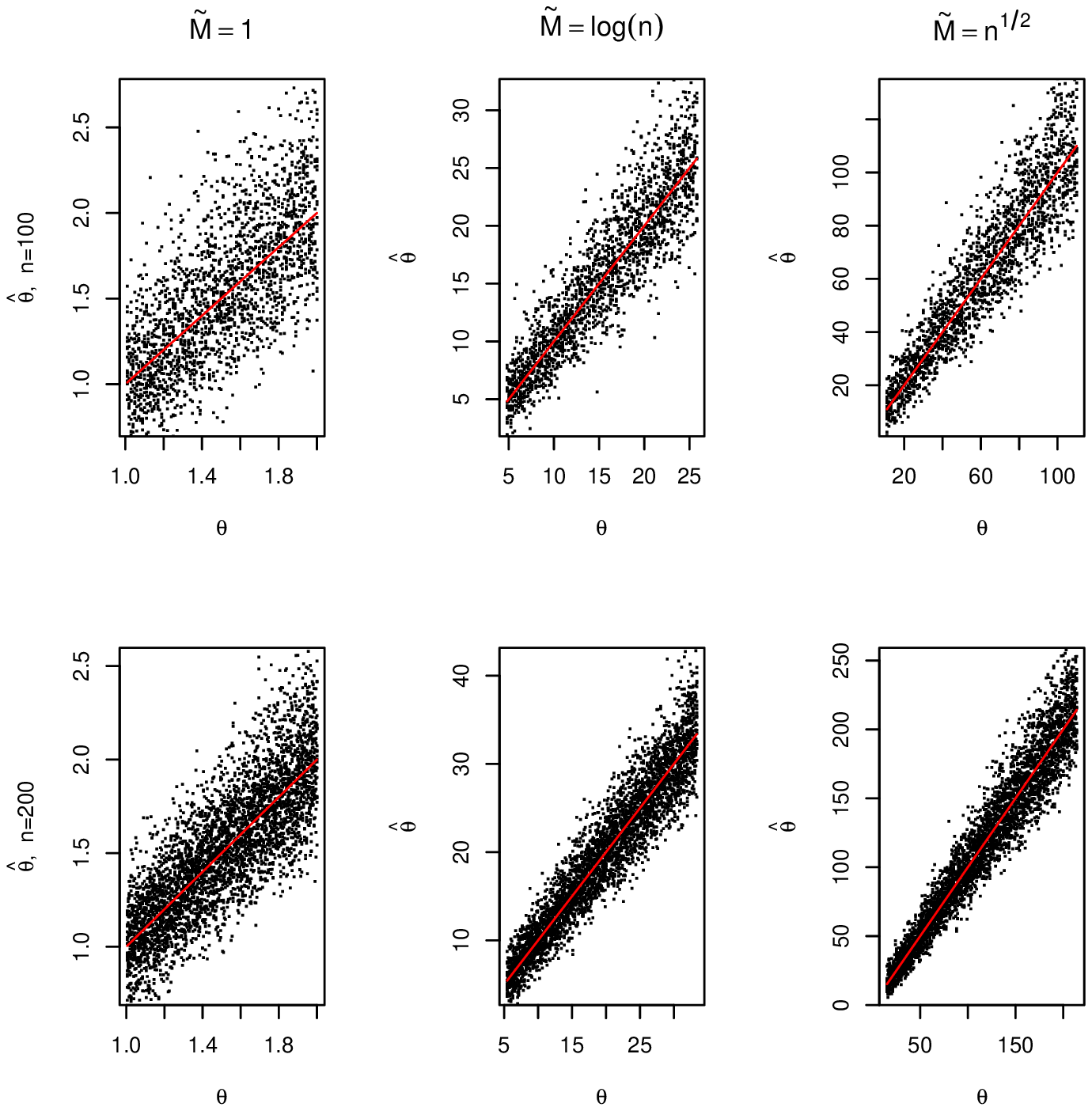}
}
\subfigure[Discrete weights]{
\includegraphics[height=4in, width=6in, angle=0]{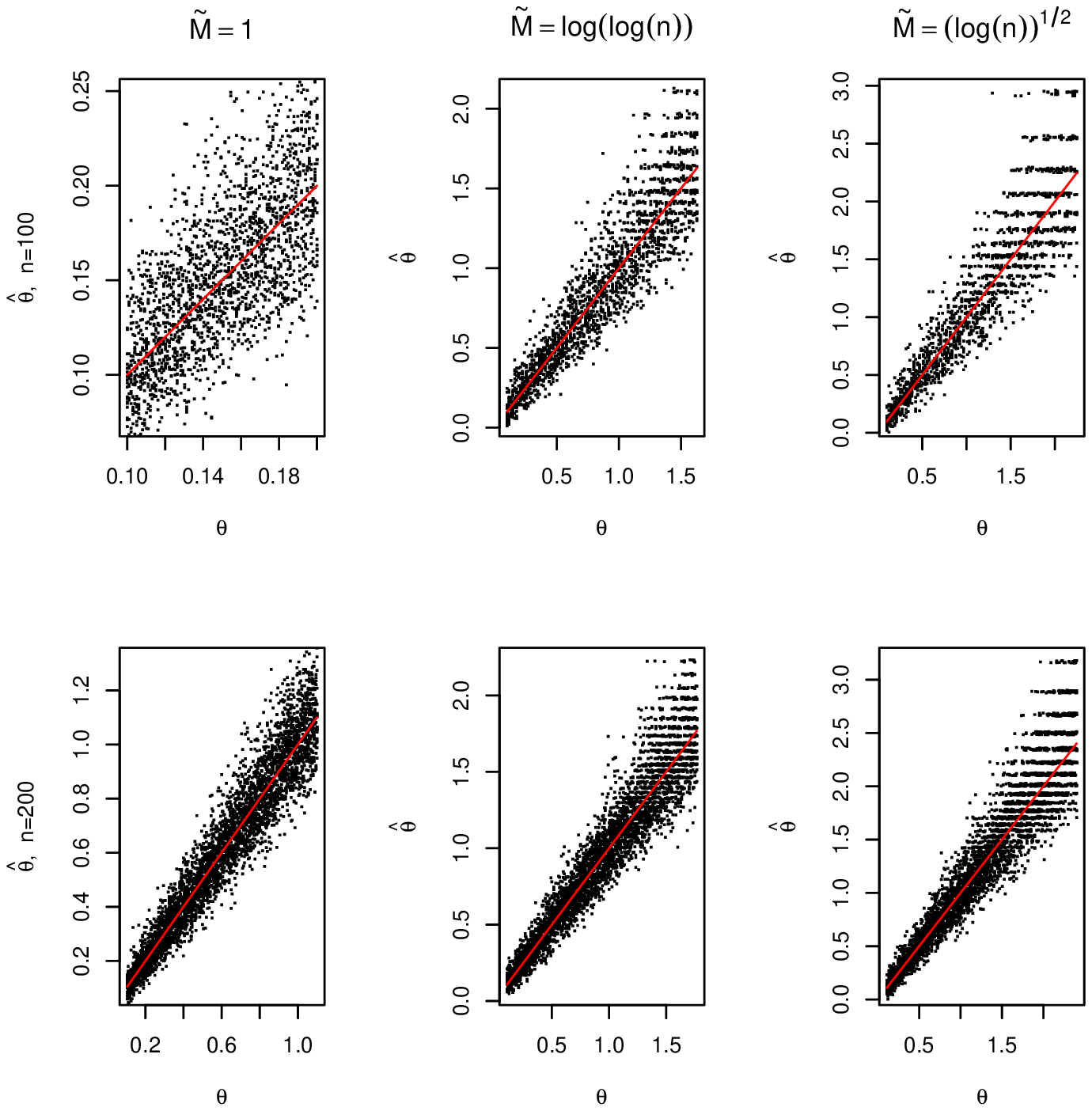}
}
\end{figure}

\begin{figure}[!htb]
\centering
\caption{The QQ plots of $\hat{v}_{ii}^{1/2}(\hat{\theta}_i-\theta_i)$ in the case of continuous weights.}
\label{figure-continuous-qq}
\subfigure[$n=100$]{
\includegraphics[ height=4in, width=6in, angle=0]{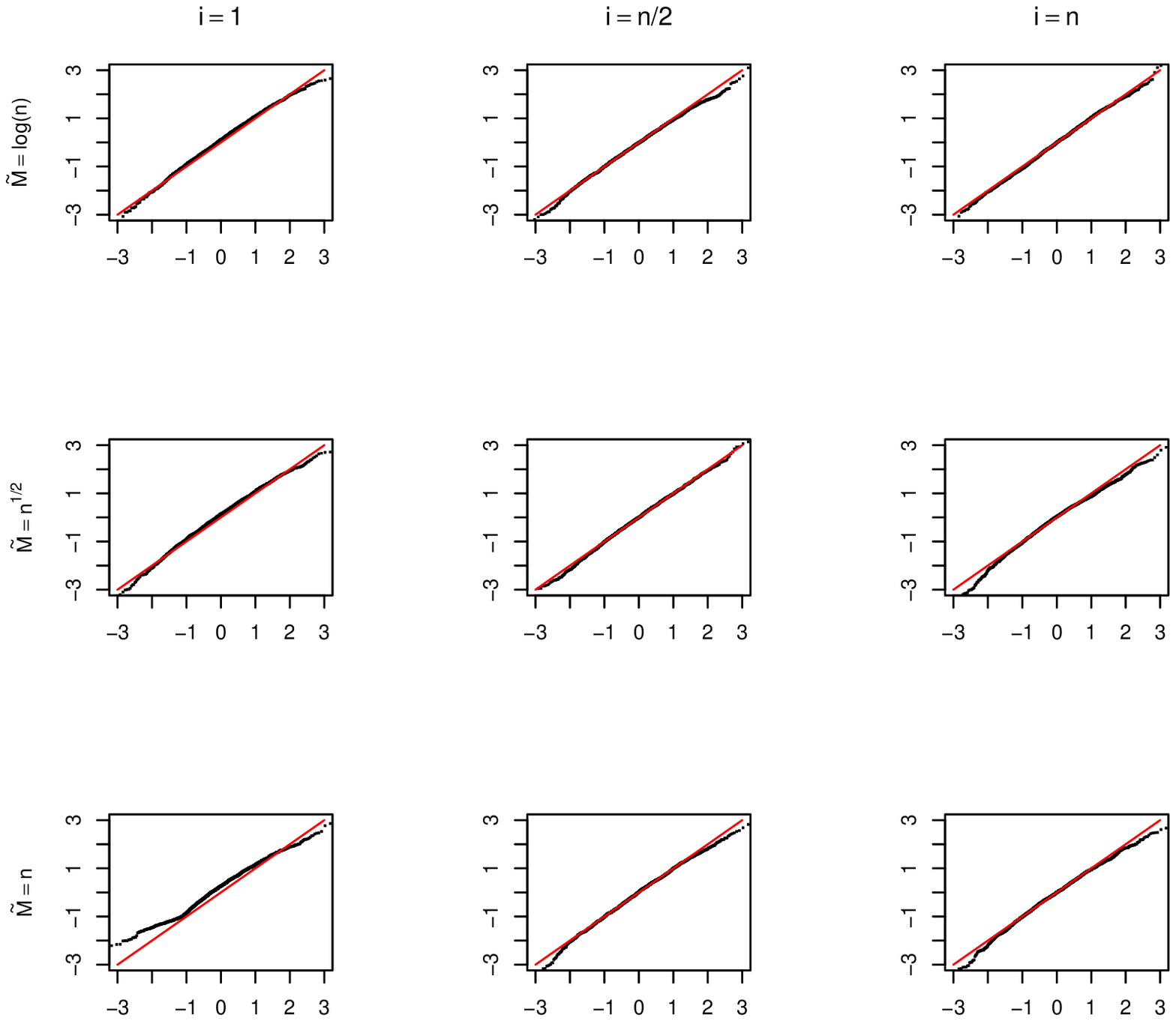}
}
\subfigure[$n=200$]{
\includegraphics[height=4in, width=6in, angle=0]{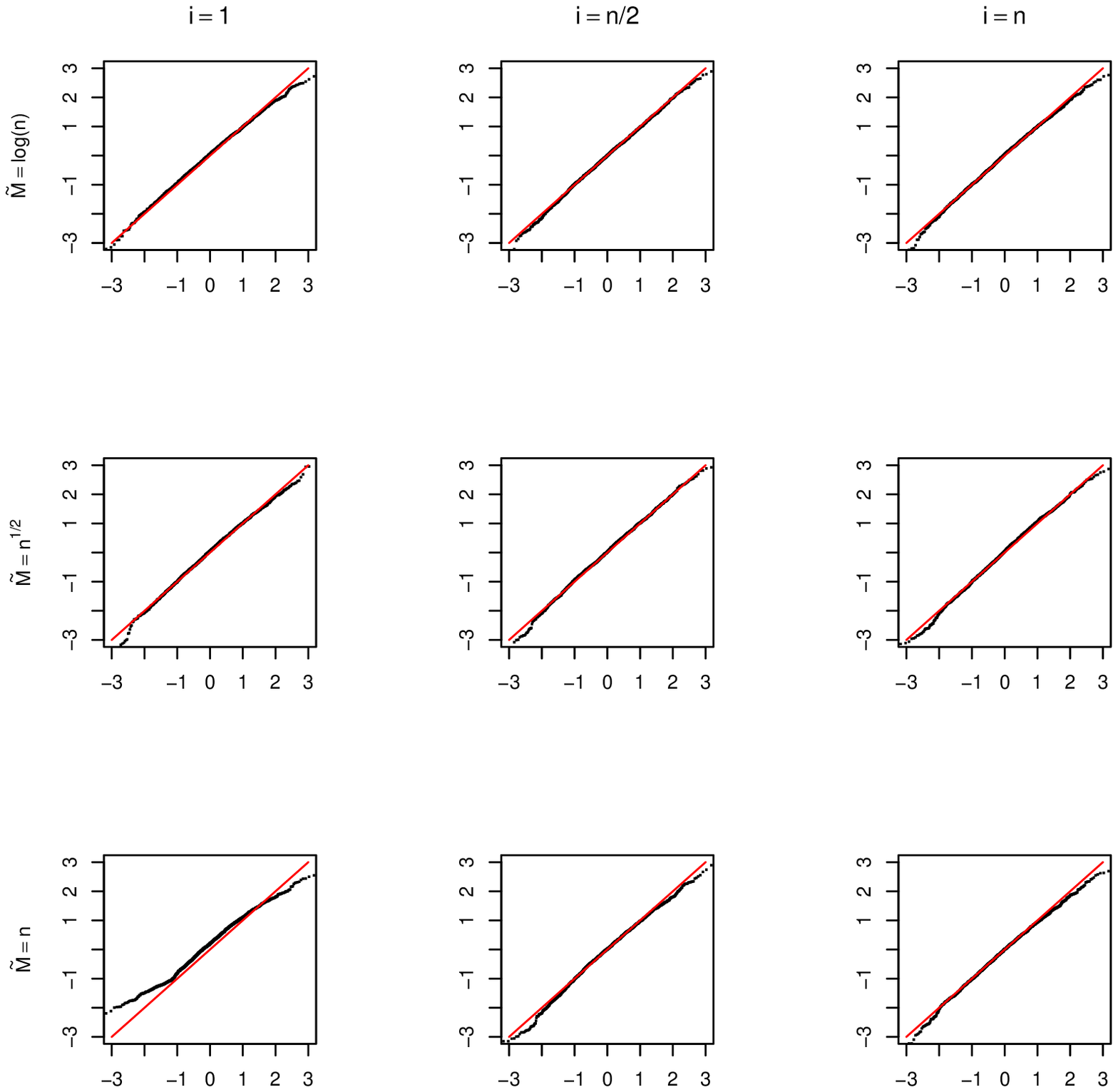}
}
\end{figure}

\newpage

\begin{figure}[!htb]
\centering
\caption{The QQ plots of $\hat{v}_{ii}^{1/2}(\hat{\theta}_i-\theta_i)$ in the case of discrete weights. }
\label{figure-discrete-qq}
\subfigure[$n=100$]{
\includegraphics[ height=4in, width=6in, angle=0]{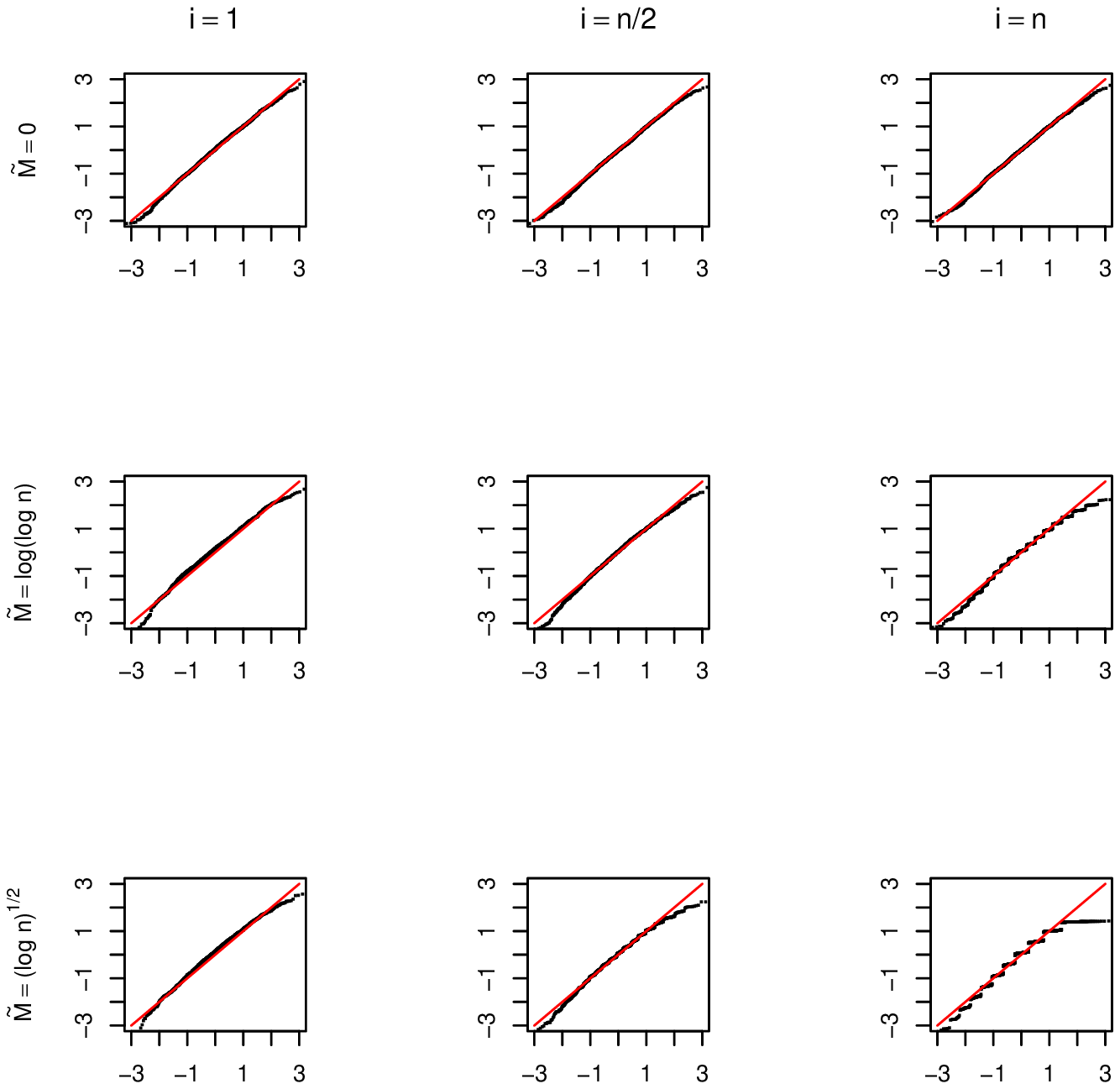}
}
\subfigure[$n=200$]{
\includegraphics[height=4in, width=6in, angle=0]{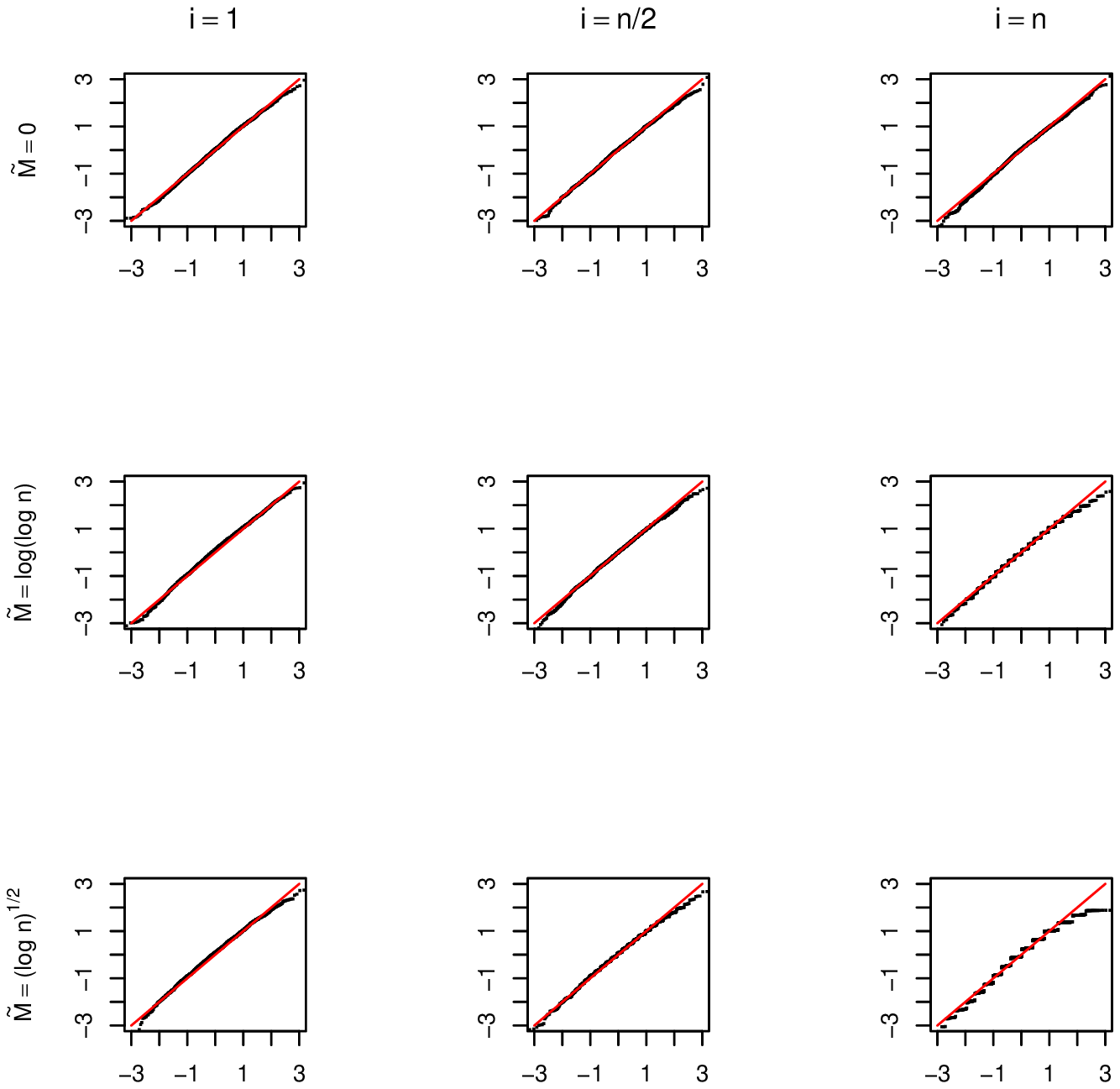}
}
\end{figure}

\newpage

\begin{table}[h]\centering
\caption{Estimated coverage probabilities of $\theta_i-\theta_j$ for pair $(i,j)$ as well as the probabilities
that the MLE does not exist (in parentheses), multiplied by $100$, and
the length of confidence intervals (in square brackets).}
\label{Table:continuous:dis}
\scriptsize
\vskip5pt
\begin{tabular}{ccccccc}
\hline
\\
\multicolumn{6}{c}{ Weighted random graphs with continuous weights}\\
\hline
\\
n       &  $(i,j)$ & $\widetilde{M}=1$ & $\widetilde{M}=\log n$ & $\widetilde{M}=n^{1/2}$ & $\widetilde{M}=n$ \\
\hline
\\
50           &$(1,50)  $&$  95.55[2.00]  $&$ 95.40[11.83]  $&$ 95.62[28.08] $&$  96.35[730.54]$\\
             &$(25,26) $&$  95.63[2.39]  $&$ 95.15[17.93]  $&$ 95.57[48.59] $&$  95.95[1889.10]$\\
             &$(49,50) $&$  95.25[2.80]  $&$ 95.64[24.03]  $&$ 95.25[69.17] $&$  96.15[2932.88]$\\

&&&&&&\\
100         &$(1,100) $&$  95.43[1.39]  $&$ 95.25[9.84]  $&$ 94.75[30.29]  $&$ 95.45[1409.94]$ \\
            &$(50,51) $&$  94.75[1.67]  $&$ 95.60[16.20]  $&$ 95.25[62.13] $&$ 95.05[5059.72]$ \\
            &$(99,100)$&$  95.55[1.95]  $&$ 94.85[22.29] $&$ 95.25[91.70]  $&$ 96.45[8060.81]$ \\

&&&&&&\\
200         &$(1,200)   $&$ 95.51[0.97] $&$ 95.45[8.21]  $&$95.45[33.93] $&$ 95.35[2750.85]$\\
            &$(100,101) $&$ 95.10[1.17] $&$ 95.05[14.28] $&$95.05[81.45] $&$ 95.35[13958.43]$\\
            &$(199,200) $&$ 95.36[1.37] $&$ 94.67[20.10] $&$95.39[123.43]$&$ 95.59[22556.82]$ \\

&&&&&&\\
\hline
\\
\multicolumn{6}{c}{ Weighted random graphs with discrete weights} \\
\hline
\\
n       &  $(i,j)$ & $\widetilde{M}=0$ & $\widetilde{M}=\log(\log n)$ & $\widetilde{M}=(\log n)^{1/2}$ & $\widetilde{M}=\log(n)$ \\
\hline
\\
50           &$(1,50)   $&$ 95.55[0.16](0) $&$ 94.37[0.56](1.35) $&$ 95.04[0.71](51.55) $&$ (100)$ \\
             &$(25,26)  $&$ 95.10[0.16](0) $&$ 96.45[1.30](1.35) $&$ 97.27[2.01](51.55) $&$ (100)$ \\
             &$(49,50)  $&$ 95.95[0.16](0) $&$ 97.52[2.23](1.35) $&$ 100.00[3.65](51.55) $&$ (100)$ \\

&&&&&&\\
100         &$(1,100) $&$  95.17[0.11](0) $&$  94.45[0.37](0.05)$&$  95.46[0.46](16.75)$&$  (100)$\\
            &$(50,51) $&$  95.15[0.11](0) $&$  95.75[0.99](0.05)$&$  95.34[1.45](16.75)$&$  (100)$\\
            &$(99,100)$&$  94.95[0.11](0) $&$  95.85[1.74](0.05)$&$  98.91[3.00](16.75)$&$  (100)$\\

&&&&&&\\
200         &$(1,200)   $&$ 95.15[0.08](0) $&$  94.68[0.26](0) $&$ 94.77[0.31](1.60) $&$ (100)$ \\
            &$(100,101) $&$ 94.85[0.08](0) $&$  95.55[0.75](0) $&$ 95.57[1.09](1.60) $&$ (100)$ \\
            &$(199,200) $&$ 95.45[0.08](0) $&$  95.62[1.33](0)  $&$ 97.51[2.28](1.60) $&$ (100)$ \\
\hline
\end{tabular}
\end{table}


\begin{thebibliography}{9}
\setlength{\itemsep}{-2pt}


\bibitem{Banavar:Maritan:Volkov:2010}
Banavar J. R., Maritan A. and Volkov I. (2010).
Applications of the principle of maximum
entropy: from physics to ecology. {\it Journal of Physics: Condensed Matter}, {\bf 22}, 063101.

\bibitem{Bickel:Chen:2009}
Bickel, P. J. and Chen, A. (2009). A nonparametric view of network models and
Newman-Girvan and other modularities. {\it Proc. Natl. Acad. Sci. USA}, {\bf 106}, 21068--21073.

\bibitem{Bickel:Chen:Levina:2012}
Bickel, P. J., Chen, A. and Levina, E. (2012). The method of moments and degree
distributions for network models. {\it Ann. Statist.}, {\bf 39}, 2280--2301.

\bibitem{Brown:1986}
Brown, L. D. (1986). {\it Fundamentals of Statistical Exponential Families with Applica-
tions in Statistical Decision Theory. Institute of Mathematical Statistics Lecture
Notes¡ªMonograph Series 9.} IMS, Hayward, CA.

\bibitem{Chatterjee:Diaconis:2011}
Chatterjee, S. and Diaconis P. (2011). Estimating and understanding exponential random 	
graph models. {\it Annals of Statistics}, Accepted, Avaible at \url{http://arxiv.org/abs/1102.2650}.

\bibitem{Chatterjee}
Chatterjee S., Diaconis P., and Sly A. (2011).
Random graphs with a given degree sequence. {\it Annals of Applied Probability}, \textbf{21}, 1400--1435.

\bibitem{Duik:Phillips:Schapire:2007}
Dud\'{i}k M., Phillips S. J. and Schapire R. E. (2007).
Maximum entropy density estimation with generalized
regularization and an application to species distribution modeling.
{\it Journal of Machine Learning Research} {\bf 8} 1217--1260.

\bibitem{Frank:Strauss:1986}
Frank, O. and Strauss D. (1986). Markov graphs. {\it Journal of the American Statistical
Association}, {\bf 81}(395), 832--842.

\bibitem{Handcock:2003}
Handcock, M. (2003). Statistical models for social networks: Inference and degeneracy.
In R. Breiger, K. Carley, and P. Pattison (Eds.), Dynamic Social Network Modeling and
Analysis: Workshop Summary and Papers.Washington, D.C.: National Academies Press.

\bibitem{Hillar:Wibisono:2013}
Hillar C. and Wibisono A. (2013). Maximum entropy distributions on graphs.
Available at \url{http://arxiv.org/abs/1301.3321}.

\bibitem{Hunter:2004}
Hunter, D. R. (2004). MM algorithms for generalized Bradley¨CTerry models. {\it The Annals of Statistics
}, {\bf 32}, 384--406.

\bibitem{Hunter:Handcock:2006}
Hunter, D. R. and  Handcock  M. S. (2006). Inference in curved exponential family models
for networks. {\it Journal of Computational and Graphical Statistics} {\bf 15}, 565--583.

\bibitem{Holland:Leinhardt:1981}
Holland, P. W. and Leinhardt, S. (1981). An exponential family of probability distributions for directed graphs.
{\it Journal of the American Statistical
Association}, {\bf 76}, 33--50.

\bibitem{Olhede:Wolfe:2012}
Olhede S. C. and Wolfe P. J. (2012). Degree-based network models. Available at
\url{http://arxiv.org/abs/1211.6537}.

\bibitem{Perry:Wolfe:2012}
Perry P. O. and Wolfe P. J. (2012). Null models for network data.
Available at \url{http://arxiv.org/abs/1201.5871}.

\bibitem{Rinaldo2013}
Rinaldo A., Petrovi\'{c} S., and Fienberg S. E. (2013). Maximum lilkelihood estimation in the $\beta$-model.
{\it Ann. Statist.} {\bf 41}(3), 1085--1110.

\bibitem{Schweinberger:2011}
Schweinberger, M. (2011). Instability, sensitivity, and degeneracy of discrete exponential
families. {\it Journal of the American Statistical Association}, {\bf 106}(496), 1361--1370.

\bibitem{Shalizi:Rinaldo:2013}
Shalizi, C. R., and Rinaldo, A. (2013). Consistency under sampling of exponential random graph models.
{\it The Annals of Statistics}, {\bf 41}, 508--535.

\bibitem{Snijders:2002}
Snijders, T. A. B. (2002). Markov chain Monte Carlo estimation of exponential random
graph models. {\it Journal of Social Structure}, {\bf 3}, 1--40.

\bibitem{Strauss:1986}
Strauss, D. (1986). On a general class of models for interaction. {\it SIAM Review}, {\bf 28}, 513--527.

\bibitem{Jordan:Wainwright:2008}
Wainwright M. and Jordan M. I. (2008). Graphical models, exponential families, and variational inference.
{\it Foundations and Trends in Machine Learning}, \textbf{1}(1--2):1--305.

\bibitem{Wu:2003}
Wu X. (2003). Calculation of maximum entropy densities with application to income distribution.
{\it Journal of Econometrics}, {\bf 115}, 347--354.



\bibitem{Yan:Xu:2013}
Yan T. and Xu J. (2013). A central limit theorem in the $\beta$-model for undirected random graphs with a diverging number of vertices.
{\it Biometrika}. {\bf 100}, 519--524.

\bibitem{Yeo:Bergue:2004}
Yeo G. and Burge C. B. (2004).
Maximum entropy modeling of short sequence motifs with applications to RNA splicing signals.
{\it Journal of Computational Biology}, {\bf 11}, 377--394.

\bibitem{Zhao:Levina:Zhu:2012}
Zhao, Y., Levina, E. and Zhu, J. (2012). Consistency of community detection
in networks under degree-corrected stochastic block models. {\it Ann. Statist.}, {\bf 40}, 2266--2292.


\end{thebibliography}
\end{document}